\documentclass[12pt]{amsart}

\usepackage{amsrefs}
\usepackage{amssymb}
\usepackage{fancyhdr}
\usepackage{xcolor}
\usepackage{bbm}
\usepackage{thmtools}
\usepackage{thm-restate}
\usepackage{hyperref}

\declaretheorem[name=Theorem,numberwithin=section]{theorem}
\newcommand{\ignore}[1]{}

\newtheorem{lemma}[theorem]{Lemma}
\newtheorem{proposition}[theorem]{Proposition}
\newtheorem*{proposition*}{Proposition}
\newtheorem{claim}[theorem]{Claim}
\newtheorem{corollary}[theorem]{Corollary}

\newtheorem{definition}[theorem]{Definition}

\newtheorem*{lemma*}{Lemma}
\usepackage{graphicx}
\usepackage{pstricks, enumerate, pst-node, pst-text, pst-plot}

\newcommand{\R}{\mathbb{R}}
\newcommand{\N}{\mathbb{N}}

\newcommand{\Z}{\mathbb{Z}}

\newcommand{\eps}{\varepsilon}
\newcommand{\union}{\bigcup}

\newcommand{\half}{{\textstyle \frac12}}

\def\inter{{\tt int}}
\def\exter{{\tt ext}}

\def\shifts{\mathfrak{S}}
\def\cT{\mathcal{T}}
\def\cS{\mathcal{S}}
\def\cF{\mathcal{F}}
\def\cR{\mathcal{R}}

\begin{document}

\title[The entropy of generic shifts]{Symbolic dynamics on amenable
  groups: the entropy of generic shifts}

\author{Joshua Frisch and Omer Tamuz}
\address{Department of Mathematics, Massachusetts Institute of
  Technology, Cambridge MA 02139, USA.}


\thanks{J.\ Frisch was supported by MIT's Undergraduate Research
  Opportunities Program. This research was partially conducted at
  Microsoft Research, New England.}

\date{\today}

\begin{abstract}
  Let $G$ be a finitely generated amenable group. We study the space
  of shifts on $G$ over a given finite alphabet $A$. We show that the
  zero entropy shifts are generic in this space, and that more
  generally the shifts of entropy $c$ are generic in the space of
  shifts with entropy at least $c$. The same is shown to hold for the
  space of transitive shifts and for the space of weakly mixing
  shifts.

  As applications of this result, we show that for every entropy value
  $c \in [0,\log |A|]$ there is a weakly mixing subshift of $A^G$ with
  entropy $c$. We also show that the set of strongly irreducible
  shifts does not form a $G_\delta$ in the space of shifts, and that
  all non-trivial, strongly irreducible shifts are non-isolated points
  in this space.
\end{abstract}

\maketitle
\tableofcontents

\section{Introduction}

Let $G$ be a countable, finitely generated amenable group, and let $A$
be a finite set of symbols. $G$ acts by shifts on $A^G$, which is
endowed with the product topology. A {\em symbolic dynamical system}
or a {\em shift} is a closed, shift-invariant subset of $A^G$. The
space of all shifts $\shifts = \shifts(G,A)$ admits a natural topology
induced from the Hausdorff topology; in this topology two shifts are
close if they coincide on large finite subsets of $G$.

An important dynamical property of a shift is its {\em entropy}, which
roughly measures the exponential growth rate of its projections on
finite sets.  We show that for every $c,\eps \geq 0$ the set of shifts
with entropy between $c$ and $c+\eps$ is dense in $\shifts_{\geq c}$,
the set of shifts with entropy at least $c$. We furthermore show that
this still holds for some interesting subsets of the space of shifts,
such as the weakly mixing shifts and the transitive shifts.

This result is novel even for the case of $G=\Z$ and $c=0$, although
in this case it admits a simpler proof. The case of $G=\Z^2$ already
seems to require all of our machinery.

We show that entropy is upper semi-continuous, from which it follows
that the shifts of entropy $c$ are dense in $\shifts_{\geq c}$, and in fact
are {\em generic}. The study of genericity in dynamical systems has a
long and fruitful history. For example,
Halmos~\cite{halmos1944general} showed that a generic measure
preserving transformation is weakly mixing, and
Rohlin~\cite{rohlin1948general} showed that it is not strongly mixing;
this constituted the first proof that there exist weakly mixing
transformations that are not strongly mixing.

An immediate consequence of the fact that the shifts of entropy $c$
are generic in $\shifts_{\geq c}$ is that they {\em exist}. Hence for
every $c \in [0,\log |A|]$ there exists a shift in $A^G$ with entropy
$c$ (and in fact, this can be achieved using weakly mixing shifts).
These results use the Baire Category Theorem and thus are not
explicitly constructive; for $G=\Z$, an explicit construction of
shifts of every entropy is given by
Weiss~\cite{weiss1970intrinsically} (and see also the forthcoming book
by Coornaert~\cite{coornaert2015topological}).

Hochman~\cite{hochman-generic} studies the space of transitive shifts
over $\Z$, but where the symbols are in the Hilbert cube. Among many
results, he shows that the zero entropy shifts are generic. Note that
the space of shifts over the Hilbert cube is different than the one we
study: for example, it is connected, while ours is zero dimensional
and in fact has isolated points. These differences induce different
generic properties of the two spaces. For example, Hochman shows that
the weakly mixing shifts are dense in the space of transitive shifts;
this is not true over finite alphabets, since there are isolated
transitive shifts that are not weakly mixing. Accordingly, our
techniques are different than Hochman's, and specifically are more
combinatorial in nature.

Our main tool are {\em quasi-tiling shifts}.  A {\em tiling} of a group
is a decomposition of its elements into disjoint finite sets, where
each set is a translate of some finite number of ``tiles''. In a {\em
  quasi-tiling} there are still only a finite number of tiles, but not
all the group elements need be covered, and some may be covered by
more than one tile. A ``good'' quasi-tiling will have few such
``errors''. Another desirable property of a quasi-tiling is that the
tiles have small boundaries.  Ornstein and
Weiss~\cite{ornstein1987entropy} show that finite subsets of amenable
groups can be arbitrarily well quasi-tiled, using tiles that have
arbitrarily small boundaries. 

Quasi-tiling shifts are closed, shift-invariant sets of quasi-tilings
of the entire group. We show that ``good'' ones exist, in the sense
that they have few errors, have tiles with small boundaries, and have
disjoint tiles. Furthermore, they are strongly irreducible and have
low entropy.

In a very recent paper by Downarowicz, Huczek and
Zhang~\cite{downarowicz2015tilings}, it is shown that amenable groups
can in fact be {\em tiled} by tiles with small boundaries, and that
there exist tiling shifts with low entropy. We use some of their
intermediate results in this paper, but not the actual tiling shifts,
since these are not guaranteed to be strongly irreducible.

\subsection{Definitions and results}
\subsubsection{The space of shifts}
Let $G$ be a countable, finitely generated amenable group. We fix
$d(\cdot,\cdot)$, a left-invariant word length metric on $G$. Let $A$
be a finite set of symbols. $A^G$ is endowed with the product topology
and with the $G$-action by shifts, given by $[g x](h) = x(g^{-1}h)$.

A closed, $G$-invariant subset $X$ of $A^G$ is called a {\em
  shift}. The space of all shifts is denoted by $\shifts =
\shifts(G,A)$ and is equipped with the Hausdorff topology, or, more
precisely, with its restriction to the shifts. As such it is a compact
Polish space.

This topology admits a simple geometrical definition.  For a finite
set $K \subset G$, a shift $X \in \shifts$, and an $x \in X$, let $x_K
\colon K \to A$ be the restriction of $x$ to $K$, and let $X_K$ be the
projection of $X$ on $K$:
\begin{align*}
  X_K = \{x_K\,:\,x \in X\}.
\end{align*}
Then a sequence of shifts $\{X^n\}_{n=1}^\infty$ converges to $X$ if
and only if
for all finite $K \subset G$ it holds that $\lim_n X^n_K = X_K$, or
equivalently that $X^n_K=X_K$ for $n$ large enough.

A subset of $\shifts$ is a $G_\delta$ if it is a countable
intersection of open sets. A set that contains a dense $G_\delta$
subset is called generic (equivalently: residual or comeagre) and its
complement is called meager. By the Baire Category Theorem, a
countable intersection of dense open sets is comeagre.

\subsubsection{F{\o}lner sets and entropy}
Let $B_r \subset G$ be the ball of radius $r$ around the origin in
$G$:
\begin{align*}
  B_r = \{g \in G\,:\,d(e,g) \leq r\}.
\end{align*}
\begin{definition}
  Let $F \subset G$ be finite.  For $r \in \N$, the $r$-boundary of
  $F$ is
  \begin{align*}
    \partial_r F = \{g \in G\,:\, g B_r \cap F \neq \emptyset \mbox{
      and } g B_r \cap (G \setminus F) \neq \emptyset\},
  \end{align*}
  and the $r$-boundary ratio of $F$ is
  \begin{align*}
    \rho_r(F) = \frac{|\partial_r F|}{|F|},
  \end{align*}
  where $|\cdot|$ is the counting measure on $G$. We say that $F$ is
  {\em $(r,\eps)$-invariant} if $\rho_r(F) \leq \eps$.
\end{definition}

A characterization of amenable groups is the existence of {\em
  F{\o}lner sequences}: a sequence $\{F_n\}$ of finite subset of $G$
is F{\o}lner if $\lim_n \rho_r(F_n) = 0$.

\begin{definition}
  The entropy of $X \in \shifts$ is given by
  \begin{align*}
    h(X) = \lim_{n \to \infty}\frac{1}{|F_n|}\cdot \log|X_{F_n}|.
  \end{align*}  
\end{definition}
That this limit exists and is independent of the choice of F{\o}lner
sequence was shown by Ornstein and Weiss~\cite{ornstein1987entropy}.

\begin{restatable}{proposition}{semicont}
  \label{prop:semi-cont-ent}
  The entropy map $h \colon \shifts \to \R^+$ is upper semi-continuous.
\end{restatable}
The (relatively straightforward) proof of this proposition appears in
Appendix~\ref{sec:semi-cont-ent}; see Lindenstrauss and
Weiss~\cite{lindenstrauss2000mean}*{Appendix 6} for a related proof
that uses similar ideas. It follows that the set of shifts with
entropy at least $c$ is a closed subset of $\shifts$, for any $c \geq
0$. We denote it by $\shifts_{\geq c}$, and denote by $\shifts_c$ the
set of shifts with entropy exactly $c$.

\subsubsection{Main results}
Our main result is the following.
\begin{theorem}
  \label{thm:main0}
  For every $c \geq 0$, $\shifts_c$ is comeagre in $\shifts_{\geq c}$.
\end{theorem}

We in fact prove a more general statement. Strongly irreducible shifts
(see definition below) have good mixing properties, and in particular
are strongly mixing.  We call a class $\cF$ of shifts {\em
  strongly-stable} if it is closed under the operations of taking
factors and of taking products with strongly irreducible shifts. For
example, the transitive shifts, the recurrent shifts, the weakly
mixing shifts, the strongly mixing shifts and the strongly irreducible
shifts are all strongly-stable classes. The first three are also
$G_\delta$ subsets of $\shifts$~\cite{hochman-generic}. Let
$\shifts^{\cF}$ be the subset of $\shifts$ that is in $\cF$, and
define $\shifts_c^{\cF}$ and $\shifts_{\geq c}^{\cF}$ similarly.
\begin{theorem}
  \label{thm:main}
  Let $\cF$ be a strongly-stable class such that $\shifts^{\cF}$ is a
  $G_\delta$ subset of $\shifts$. Then for every $c \geq 0$,
  $\shifts_c^\cF$ is comeagre in $\shifts_{\geq c}^\cF$.

  More generally, when $\shifts^{\cF}$ is not necessarily a
  $G_\delta$, for every $c\geq 0$ and $\eps >0$, the subset of shifts
  in $\shifts_{\geq c}^\cF$ with entropy in $[c,c+\eps)$ is dense in
  $\shifts_{\geq c}^\cF$.
\end{theorem}
Since the class of all shifts is also strongly-stable,
Theorem~\ref{thm:main} implies Theorem~\ref{thm:main0}. Note again
that the weakly mixing shifts and the transitive shifts both form a
$G_\delta$~\cite{hochman-generic}.

Theorem~\ref{thm:main} admits a number of interesting immediate
corollaries.
\begin{corollary}
  For every $c \in [0,\log |A|]$ there exists a weakly mixing shift
  $X \subseteq A^G$ with $h(X) =c$.
\end{corollary}

It is not known whether the set of strongly mixing shifts is a
$G_\delta$ (see Hochman~\cite{hochman-generic}). However, the next
result follows directly from Theorem~\ref{thm:main}.
\begin{corollary}
  The set of strongly irreducible shifts is not a $G_\delta$.
\end{corollary}
This is a consequence of the fact that non-trivial (i.e., of
cardinality greater than one) strongly irreducible shifts have positive
entropy. A natural question is: are there uncountably many strongly
irreducible shifts? Another consequence on strongly irreducible shifts
is the following corollary, which follows immediately from the second
part of Theorem~\ref{thm:main} and the fact that non-trivial strongly
irreducible shifts have positive entropy.
\begin{corollary}
  All non-trivial strongly irreducible shifts are non-isolated points
  in $\shifts$.
\end{corollary}

Another direct consequence is the following.
\begin{corollary}
  The isolated points in $\shifts^\cF_{\geq c}$ all have entropy $c$.
\end{corollary}
Since any isolated point in $\shifts_{\geq c}$ has to be a shift of finite
type, and since there are only countably many such shifts, it also
follows that
\begin{corollary}
  For all but a countable set of values of $c$, the space
  $\shifts_{\geq c}$ has no isolated points, and is therefore
  homeomorphic to the Cantor space.
\end{corollary}
We furthermore show that for a dense set of entropy values
$\shifts_{\geq c}$ indeed has isolated points.
\begin{restatable}{theorem}{isolated}
  \label{thm:isolated}
  There is a dense subset $C \subset [0,\log|A|]$ such that
  for each $c \in C$ there exists a strongly irreducible shift $X
  \subseteq A^G$ with $h(X) = c$, and where $X$ is an isolated point
  in $\shifts_{\geq c}$.
\end{restatable}

\subsubsection{Quasi-tiling shifts}
To prove Theorem~\ref{thm:main} we show the existence of good {\em
  strongly irreducible quasi-tiling shifts}. A quasi-tiling of a
countable group $G$ is a partial covering of $G$ by translates of a
finite number of ``tiles'' or finite subsets of $G$. A quasi-tiling is
``good'' if there are few overlaps and uncovered regions. See Ornstein
and Weiss~\cite{ornstein1987entropy}, and also Ceccherini-Silberstein
and Coornaert~\cite{ceccherini2010cellular}*{Section 5.6}, who use a
somewhat different formulation. Note that usually tile-translates in
quasi-tilings are allowed to overlap. We, however, will only use {\em
  disjoint} quasi-tilings, where tile-translates do not overlap. 

Strong irreducibility is a strong form of mixing; in particular
stronger than strong mixing. We now define strong irreducibility.
\begin{definition}
  A shift $X$ is {\em strongly irreducible} if there exists an $r \in
  \N$ such that, for any two finite subsets $K, H\in G$ which satisfy
  $d(k,h) > r$ for all $k \in K$ and $h \in H$, and any $x,y \in X$,
  there exists a $z \in X$ with $z_K = x_K$ and $z_H =
  y_H$.
\end{definition}

\begin{definition}
  A {\em tile set} $\cT = (T_1,\ldots,T_n)$ is a finite collection of
  finite subsets of $G$, each of which includes the identity.

  Given an $r>0$ and $\eps > 0$, a tile set $\cT$ is said to be
  $(r,\eps)$-invariant if each $T_i$ is $(r,\eps)$-invariant.

  We denote by $r(\cT) = \max\{d(e,g)\,:\, g \in T_i,i=1,\ldots,n\}$
  the radius of $\cT$. This is the radius of the smallest ball that
  contains all the tiles in $\cT$.
  
  A {\em tile-translate} of $T_i$ is a set of the form $h T_i$, for some $h
  \in G$. 
\end{definition}

We next define a $\cT$-tiling, which is a covering of $G$ by
disjoint translates of the tiles in $\cT$.
\begin{definition}
  Let $\cT$ be a tile set. $T \colon G \to \cT \cup \emptyset$ is
  called a $\cT$-{\em tiling} if $g T(g) \cap h T(h) = \emptyset$ for
  all $g \neq h \in G$, and if, for each $g\in G$, there exists a unique
  $h \in G$ with $g \in h T(h)$.
\end{definition}
Here $h$ can be thought of as the ``corner'' of the tile-translate
$h T(h)$. As mentioned above, recently Downarowicz, Huczek and
Zhang~\cite{downarowicz2015tilings} showed that all amenable groups
admit good tilings. We will not use this result.

A $\cT$-quasi-tiling is much like a tiling, except that not all of $G$
need be covered. We will still require the tile-translates to be
disjoint (thus parting from the usual terminology), and will sometimes
explicitly refer to these quasi-tilings explicitly as disjoint
quasi-tilings.
\begin{definition}
  Let $\cT$ be a tile set. $T \colon G \to \cT \cup \emptyset$ is
  called a (disjoint) $\cT$-{\em quasi-tiling} if $g T(g) \cap h T(h)
  = \emptyset$ for all $g, h \in G$.

  A $g \in G$ is said to be a $T$-{\em error} if it is not an element
  of any tile-translate $g T(g)$.
\end{definition}
In a slight abuse of notation, we denote the set of (disjoint)
$\cT$-quasi-tilings by $\cT^G$.

A ``good'' quasi-tiling $T \in \cT^G$ will have few errors. The following
definition is a useful way to quantify that.
\begin{definition}
  Let $T \in \cT^G$ be a quasi-tiling. Denote by $e(T,F)$ be the
  number of $T$-errors in $F \subset G$.  The {\em error density} of a
  quasi-tiling $T$ is said to be at most $\eps$ if there exists some
  $\delta > 0$ such that
  \begin{align*}
    \frac{e(T,F)}{|F|} \leq \eps
  \end{align*}
  whenever $F$ is $(1,\delta)$-invariant. 
\end{definition}

\subsubsection{Shifts of quasi-tilings}
Given a tile set $\cT$, a {\em quasi-tiling shift} is simply a subshift
of $\cT^G$. 

\begin{definition}
  The {\em error density} of a $\cT$-quasi-tiling shift $Q$ is said to be
  at most $\eps$ if there exists some $\delta > 0$ such that
  \begin{align*}
    \frac{e(T,F)}{|F|} \leq \eps
  \end{align*}
  for all $T \in Q$ and whenever $F$ is $(1,\delta)$-invariant.
\end{definition}

The following theorem shows that ``good'' disjoint quasi-tiling shifts
exist. Its proof is straightforward, given the results of Ornstein and
Weiss~\cite{ornstein1987entropy} and a recent paper of Downarowicz,
Huczek and Zhang~\cite{downarowicz2015tilings}; the only missing
ingredient is strong irreducibility.
\begin{theorem}[Ornstein and Weiss~\cite{ornstein1987entropy},
  Downarowicz, Huczek and Zhang~\cite{downarowicz2015tilings}]
  \label{thm:quasi-tiling-shifts-intro}
  For every $r,\eps$ there exists an $(r,\eps)$-invariant tile set
  $\cT$ and a shift $Q \subset \cT^G$ that has the following properties:
  \begin{enumerate}
  \item The error density of $Q$ is at most $\eps$.
  \item $h(Q) < \eps$.
  \item $Q$ is strongly irreducible.
  \end{enumerate}
\end{theorem}

The following last definition of this section will be useful.
\begin{definition}
  A tile set $\cT$ is $\eps$-{\em good} if there exists a quasi-tiling
  shift $Q \subseteq \cT^G$ that satisfies the conditions of
  Theorem~\ref{thm:quasi-tiling-shifts-intro} with $r=1$.
\end{definition}

\subsection{Organization and notation}
The remainder of the paper proceeds as follows. In
Section~\ref{sec:quasi-tiling} we show the existence of good
quasi-tiling shifts.  In Section~\ref{sec:low-ent} we show
that every shift can be approximated by a shift with entropy close to
zero. We use this construction to prove in Section~\ref{sec:fixed-ent}
that every shift can be approximated by shifts of any lower entropy,
and then prove out main theorem. Finally, in
Section~\ref{sec:topology} we show that for a dense set of entropy
values $c$ the space $\shifts_{\geq c}$ has isolated points.

We provide below an overview of the notation and nomenclature used in
this paper.

\vspace{0.2in}
\begin{tabular}{l| l}
  $G$&Finitely generated amenable group\\
  $A$&Finite alphabet\\
  $B_r$&The ball of radius $r$ around the identity in $G$\\
  $\partial_r K$&The boundary of radius $r$ of a set $K \subseteq G$\\
  $\rho_r K$&The ratio between $\partial_r K$ and $|K|$\\
  $(r,\eps)$-invariant& A set $F$ satisfying $\rho_r F \leq \eps$\\
  $h(X)$& The entropy of a shift $X$\\
  $\shifts$&The space of subshifts of $A^G$\\
  $\shifts_c$&The space of subshifts of $A^G$ with entropy $c$\\
  $\shifts_{\geq c}$&The space of subshifts of $A^G$ with entropy $c$
  or higher\\
  $\cF$& A strongly stable class of shifts\\
  $\shifts^\cF$&The space of subshifts of $A^G$ belonging to $\cF$\\
  tile&A finite subset of $G$ containing the identity\\
  tile-translate&A set of the from $h T_i$ where $T_i$ is a tile and
  $h \in G$\\
  $X,Y$& Shifts, or closed, shift-invariant subsets of $A^G$\\
  $X_K$& The projection of $X$ to $K \subseteq G$\\
  $x$& Element of a shift $X$\\
  $x_K$& The projection of $x$ to $K \subseteq G$\\
  $\cT$&A tile set, or a finite set of tiles\\
  $\eps$-good tile set& A tile set allowing good quasi-tilings (as per
  Theorem~\ref{thm:quasi-tiling-shifts-intro})\\
  $r(\cT)$&The radius of the tile set $\cT$\\
  $T,S$ &A quasi-tiling\\
  $\cT^G$&The set of disjoint $\cT$-quasi-tilings\\
  $M^\cT$&The set of maximal disjoint $\cT$-quasi-tilings\\
  $e(T,K)$&The number of errors of $T$ on a
  finite $K \subset G$\\
\end{tabular}

\section{Quasi-tiling shifts}
\label{sec:quasi-tiling}

Let $\cT$ and $\cS$ be tile sets, and let $T$ and $S$ be $\cT$- and
$\cS$-quasi-tilings, respectively. A natural way of combining $T$ and
$S$ to a new (disjoint) quasi-tiling is to add to $T$ any
tile-translate of $S$ that is disjoint from all the tile-translates of
$T$. Formally, we define the map
\begin{align*}
  \psi \colon \cT^G \times \cS^G \to \left(\cT \cup \cS\right)^G
\end{align*}
by
\begin{align*}
  [\psi(T,S)](g) =
  \begin{cases}
    T(g)&\mbox{if }T(g) \neq \emptyset\\
    S(g)&\mbox{if }S(g) \neq \emptyset\mbox{ and }g S(g) \cap \cup_h h
    T(h) = \emptyset\\
    \emptyset&\mbox{otherwise}.
  \end{cases}
\end{align*}
It is easy to see that this map is continuous and commutes with the shift;
continuity follows from the fact that whether or not $g S(g) \cap
\cup_h h T(h) = \emptyset$ depends only on the values of $S$ and $T$
within distance $r(\cT)+r(\cS)$ of $g$.

We extend $\psi$ to a function on a product of more than two
quasi-tilings, as follows. Let
\begin{align*}
  \psi \colon \cT_1^G \times \cdots \times \cT_n^G \to \left(\cT_1
    \cup \cdots \cup \cT_n\right)^G
\end{align*}
be defined recursively by
\begin{align*}
  \psi(T_1,\ldots,T_n) = \psi(\psi(T_1,\ldots,T_{n-1}), T_n).
\end{align*}
That is, we start with $T_1$, add to it any tile in $T_2$ that is
disjoint, add to that any tile in $T_3$ that is disjoint, etc. By the
same reasoning used above, $\psi$ here is still continuous and
commutes with the shift.

Given a tile set $\cT$, the $\cT$-quasi-tilings can be ordered by
inclusion; namely, $T' \geq T$ if $T(g) \neq \emptyset$ implies $T'(g)
= T(g)$. By a straightforward application of Zorn's lemma, maximal
disjoint $\cT$-quasi-tilings exist. It is clear that these maximal
quasi-tilings form a closed, shift-invariant set, and are hence a
subshift of $\cT^G$, which we denote by $M^\cT$.

Given tile sets $\cT_1,\ldots,\cT_n$, we can apply $\psi$ to
$M^{\cT_1} \times \cdots \times M^{\cT_n}$. The result will be a
subshift of $\left(\cT_1 \cup \cdots \cup \cT_n\right)^G$, since
$\psi$ is continuous and commutes with the shift.

The following theorem is a more detailed restatement of
Theorem~\ref{thm:quasi-tiling-shifts-intro}.
\begin{theorem}
  \label{thm:quasi-tiling-shifts}
  For every $r,\eps$ there exists an $(r,\eps)$-invariant tile set
  $\cT = \cT_1 \cup \cdots \cup \cT_n$ such that $Q = \psi(M^{\cT_1}
  \times \cdots \times M^{\cT_n})$ has the following properties:
  \begin{enumerate}
  \item The error density of $Q$ is at most $\eps$.
  \item $h(Q) < \eps$.
  \item $Q$ is strongly irreducible.
  \end{enumerate}
\end{theorem}
This construction is based on (and is nearly identical to) the
original construction of Ornstein and Weiss; they however use
$\eps$-disjoint quasi-tilings rather than disjoint ones, and prove a
version of (1) for finite subsets of $G$. The complete proof of (1)
and (2) appears in~\cite{downarowicz2015tilings}*{Lemmata 4.1 and
  4.2}. It therefore remains to be shown that $Q$ is strongly
irreducible. To show this, it is enough to show that each $M^{\cT_i}$
is strongly irreducible, since $Q$ is formed from these shifts by
taking products and factors, and strong irreducibility is closed under
these operations. This is done in the following lemma, which thus
concludes the proof of Theorem~\ref{thm:quasi-tiling-shifts}.
\begin{lemma}
  Let $\cS$ be a tile set. Then $M^\cS$ is strongly irreducible.
\end{lemma}
\begin{proof}
  Let $r = 4r(\cS)$. Fix two finite subsets $K, H\in G$ which satisfy
  $d(k,h) > r$ for all $k \in K$ and $h \in H$. Choose any $T,S \in
  M^\cS$. We will prove the claim by showing that there exists an $R
  \in M^\cS$ with $R_K = T_K$ and $R_H = S_H$.

  Let $K' = K \cup \partial_{r(\cS)} K$ and likewise let $H' = H
  \cup \partial_{r(\cS)} H$. Then $K'$ and $H'$ are disjoint, since
  $d(K,H) > 4r(\cS)$, and furthermore the tile-translates whose
  corners are in $K' \cup H'$ are disjoint. Consider the set of
  $\cS$-quasi-tilings that include all the tile-translates that $T$
  has on $K'$ and that $S$ has on $H'$.  By Zorn's lemma there exists
  a maximal (with respect to inclusion) element in this set, which we
  will call $R$. By definition, $R$ includes all the tile-translates
  that $T$ has on $K$ and that $S$ has on $H$. It thus remains to be
  shown that (1) $R$ does not include any additional tile-translates
  on $K \cup H$, and (2) $R$ is an element of $M^\cS$.

  To see (1), note that any tile that can be added to $R$ in $K$
  (while keeping it a disjoint quasi-tiling) can also be added to $T$,
  since $R$ includes all the tile-translates of $T$ that are close
  enough to $K$ to intersect tile-translates in $K$; those are the
  precisely the tile-translates in $K'$. But it is not possible to add
  more tile-translates to $T$, since it is maximal. The same applies
  to adding to $H$.

  Finally, to see (2), note that $R$ is maximal in the set of
  quasi-tilings that are greater (again with respect to inclusion)
  than the quasi-tiling which only includes the tiles of $T$ on $K'$
  and the tiles of $S$ on $H'$. Hence it is also maximal in the set of
  all quasi-tilings.
\end{proof}

Rather than using Theorem~\ref{thm:quasi-tiling-shifts} as stated for
our work, we will need a corollary which we are now almost in a
position to prove. We will need the following additional definition.
\begin{definition}
  Let $T$ be a $\cT$-quasi-tiling, and fix $r \in \N$. The $r$-{\em
    interior} of $T$ is the set
  \begin{align*}
    \bigcup_{g \in G}g T(g) \setminus \bigcup_{g \in G}\partial_r g T(g),
  \end{align*}
  which is the union of all the tile-translates in $T$, from which is
  removed the union of the boundaries (of radius $r$) of all the
  tile-translates. The $r$-{\em exterior} of $T$ is the complement of
  the interior; alternatively, it is the set of elements of $G$ that
  are not covered in $T$, or are within distance $r$ of the boundary
  of a tile-translate in $T$.


  A $\cT$-quasi-tiling shift $Q$ has {\em $r$-exterior density} at
  most $\eps$ if there exists a $\delta >0$ such that for any finite,
  $(1,\delta)$-invariant $F \subset G$, and any $T \in Q$ it holds
  that the intersection of the $r$-exterior of $T$ with $F$ is of size
  at most $\eps|F|$.
\end{definition}

With these definitions in place we are ready to state our
corollary.
\begin{theorem}
  \label{cor:quasi-tilings}
  For all $\eps,r$ there exist a tile set $\cT$ and a strongly
  irreducible $\cT$-quasi-tiling shift $Q$ with $r$-exterior density at
  most $\eps$, and such that $h(Q)<\eps$.
\end{theorem}
\begin{proof}
  Fix $r,\eps$. Let $\cT$ be a
  tile set and $Q$ a $\cT$-quasi-tiling shift with the following
  properties.
  \begin{enumerate}
  \item $\cT$ is $(r,\eps/2)$-invariant. 
  \item The error density of $Q$ is at most $\eps/2$.
  \item $h(Q) < \eps$.
  \end{enumerate}
  The existence of such $\cT$ and $Q$ is guaranteed by
  Theorem~\ref{thm:quasi-tiling-shifts}. It is straightforward to show
  that the $r$-exterior density of $Q$ is at most $\eps$.
\end{proof}

\section{Low entropy approximation}
\label{sec:low-ent}

In this section we prove the following proposition.
\begin{proposition}
  \label{prop:low-entropy-approx}
  Let $\cF$ be a strongly-stable class of shifts. For every $\eps>0$,
  the set of shifts in $\cF$ with entropy in $[0, \eps)$ is dense in
  $\shifts^\cF$.
\end{proposition}

\subsection{Overview}

To prove this proposition we construct shifts that are in $\cF$, have
entropy in $[0,\eps)$, and arbitrarily well approximate a given shift
$X$.

To construct an approximating shift $Y$ we use strongly irreducible
quasi-tiling shifts. Fix $r > 0$, let $T = (T_1,\ldots,T_n)$ be a tile
set, and let $Q$ be a $\cT$-quasi-tiling shift. Let $T \in Q$, and
recall that the $r$-exterior of $T$ is the subset of $G$ that is
within radius $r$ of the boundary of a tile-translate, or is not
covered by a tile-translate.

The idea behind the construction of $Y$ is the following.  Given a
configuration $x \in X$, we first choose a quasi-tiling $T \in Q$,
and map the pair $(x,T)$ to a $y' \in Y' \subset A^G$, as
follows. On the $r$-exterior of $\tau$ we let $y'(g) = x(g)$, and
elsewhere we let $y'(g)=\emptyset$, for some fixed symbol $\emptyset
\in A$. We then map the pair $(y',T)$ to a $y \in Y$, by leaving
$y(g) = y'(g) = x(g)$ on the $r$-exterior of $T$, and choosing
$y(g)$ elsewhere using a ``completion map'' that is locally compatible
with $X$, but that is completely determined by the values of $y$ on
the exterior. Both of the maps $X \times Q \mapsto Y'$ and $Y' \times
Q \mapsto Y$ are continuous and commute with the shift, and therefore
both $Y'$ and $Y$ are shifts. The entropy of $Y'$ is controlled by the
fact that the number of exterior points is low. The entropy of $Y$ is
bounded by $h(Y')+h(Q)$, and hence is also low. $Y$ is a good
approximation of $X$ since we can take $r$ to be large, and $x$ and
$y$ agree on the $r$-exterior of $T$.

\subsection{Deletion and local completion maps}

Let $K$ be a subset of $G$, and let
$e \colon K \to \{\exter,\inter\}$ be a labeling of $K$ into exterior
and interior points. Let $\emptyset \in A$ be an arbitrary
``distinguished'' symbol. Given $y \in A^K$, we define $y \cdot e \in
A^G$ by
\begin{align}
  \label{eq:exter-mult}
  [y \cdot e](g) =
  \begin{cases}
    y(g)& \mbox{ if }e(g) = \exter\\
    \emptyset&\mbox{ otherwise.}
  \end{cases}
\end{align}
Intuitively, multiplying $y$ by $e$ leaves the configuration unchanged
on the exterior of $K$, and substitutes the $\emptyset$ symbol on the
interior.

Fix $r \in \N$ for the reminder of this section; this will be a
parameter in the various constructions that follow. Given a
$\cT$-quasi-tiling $T \in \cT^G$, we let $e^T \in \{\inter,\exter\}^G$
be given by $e^T(g) = \exter$ for any $g$ on the $r$-exterior of $T$,
and $e^T(g) = \inter$ elsewhere.  As can be easily verified, the map
$T \mapsto e^T$ is a continuous and shift-equivariant map.

We next define the ``deletion'' map $\delta \colon A^G \times
\cT^G \to A^G$. Given a configuration $y \in A^G$ and a
quasi-tiling $T$, $\delta$ outputs a configuration in which the
$r$-exterior is left unchanged, and the interior of the tiles is
``deleted'', leaving there the symbol $\emptyset \in A$. Formally,
\begin{align*}
  \delta(y,T) = y \cdot e^T =
  \begin{cases}
    y(g)& \mbox{ when }e^T(g) = \exter\\
    \emptyset&\mbox{ otherwise.}
  \end{cases}
\end{align*}
The product $y \cdot e^T$ is given in the sense of~\eqref{eq:exter-mult}.

Note that $\delta$ is continuous and shift-equivariant, since the map
$T \mapsto e^T$ is. Hence the image $\delta(X \times
\cT^G)$ is a shift, as is $\delta(X \times Q)$, for any
quasi-tiling shift $Q$. We show that the entropy of this shift can be
controlled, using an appropriate choice of $Q$.
\begin{proposition}
  \label{prop:delta-low-ent}
  Let the $r$-exterior density of $Q$ be at most $\eps$. Then
  \begin{align*}
    h(\delta(X \times Q)) \leq \eps\log \frac{3|A|}{\eps},
  \end{align*}
  where $X \subseteq A^G$.
\end{proposition}
\begin{proof}
  By the definition of $\delta$, on every sufficiently invariant
  finite $F \subset G$, and every $y \in \delta(X \times Q)$ it holds
  that $y_F$ is equal to $\emptyset$ on all but at most $\eps|F|$ of the
  elements of $F$. Hence, by Lemma~\ref{lem:combo-low-ent}, the
  entropy of $h(\delta(X \times Q))$ is at most $\eps\log
  \frac{3|A|}{\eps}$.
\end{proof}


Given a shift $X$ and a finite $K \subset G$, a {\em local
  $(X,K)$-completion map} is a function $m \colon
 A^K \times \{\inter,\exter\}^K \to A^K$ with the following
properties. Let $(y,e) \in A^K \times \{\inter,\exter\}^K$,
and let there exist some $x \in X_K$ such that $y \cdot e = x
\cdot e$ (i.e, $x$ and $y$ agree on the exterior). Then
\begin{enumerate}
\item $m(y,e) \in X_K$.
\item $m(y,e) \cdot e = y \cdot e$.
\item $m(y \cdot e,e) = m(y,e)$.
\end{enumerate}
Intuitively, the map $m$ completes any configuration that is compatible
with $X_K$ on the exterior of $K$ to a complete configuration in
$X_K$; that is the first property. The second property ensures that
this is indeed a completion of the configuration on the exterior;
$m(y,e)$ and $y$ agree on the exterior. The third property ensures
that the completion is independent of the configuration in the
interior; $m(y,e)$ only depends on $y(g)$ if $f(g) = \exter$.

The following claim formalizes the idea that for every $x \in X$,
there exists a local $(K,X)$-completion map that will output $x_K$
whenever its input matches $x$ on the exterior. Its proof is
straightforward, if tedious, and we omit it.
\begin{claim}
  \label{claim:complete-x}
  For every finite $K \subset G$ and $x \in X_K$ there exists a local
  $(X,K)$-completion map $m$ such that $m(y,e) = x$ for all $e \in
  \{\inter,\exter\}^K$ and $y \in A^K$ such that $y \cdot e = x \cdot
  e$.
\end{claim}
Any such local $(K,X)$-completion function $m$ is said to be {\em
  $(x,K)$-compatible.}

Note that a particular implication of this claim is that for every $K$
there exists a local $(X,K)$-completion map.

Let $T_i$ be a tile in $\cT$. We will consider local
$(X,T_i)$-completion maps; these will complete configurations on the
exterior of the tile to the rest of the tile. They exist by
Claim~\ref{claim:complete-x}. In a slight abuse of notation, we will
apply a local $(X,T_i)$-completion map $m \colon
 A^{T_i} \times \{\inter,\exter\}^{T_i} \to A^{T_i}$ to any tile-translate
$h T_i$ in the natural way. That is, the function $m' \colon
A^{h T_i} \times \{\inter,\exter\}^{h T_i} \to A^{h T_i}$ given by
\begin{align*}
  m'(y,e) = h\big[m(h^{-1}y, h^{-1}e)\big]
\end{align*}
will also be denoted by $m$.

\subsection{Global completion maps}

Let $(m_1,\ldots,m_n)$ be local $(X,T_i)$-completion maps of each of
the tiles in the tile set $\cT=(T_1,\ldots,T_n)$. We would like to
define a global completion map $M \colon A^G \times \cT^G \to A^G$
that applies $m_i$ to each $T_i$-tile-translate separately, after
calculating its exterior.

\begin{definition}
  Given a configuration $y \in A^G$ and a quasi-tiling $T \in
  \cT^G$, define the global completion map $M(y,T)$ as follows.
  For $g$ that is not in any tile-translate, set
  \begin{align}
    \label{eq:M-def-2}
    [M(y,T)](g) = y(g).
  \end{align}
  For any $T$ tile-translate $h T_i$ set
  \begin{align}
    \label{eq:M-def-3}
    M(y,T)_{h T_i} = m_i(y_{h T_i},e^T_{h T_i}).
  \end{align}
\end{definition}

It follows that
\begin{align}
  \label{eq:M-def-1}
  [M(y,T)](g) = y(g) \quad \mbox{ whenever } e^T(g)=\exter.  
\end{align}

The next claim follows immediately from the local nature of this
definition.
\begin{claim}
  \label{clm:m-cont}
  $M$ is continuous and shift-equivariant.
\end{claim}
Hence $Y = M(X \times Q)$ is a subshift of $A^G$. Furthermore, since
$Q$ is strongly irreducible, then $Y \in \shifts^\cF$ whenever $X \in
\shifts^\cF$.

We next show that $M$ inherits the three properties of the local maps
$m_1,\ldots,m_n$.
\begin{claim}
  \label{clm:M-props}
  Fix $(y,T) \in A^G \times \cT^G$ such that $y \cdot e^T = x
  \cdot e^T$ for some $x \in X$. Then
  \begin{enumerate}
  \item $M(y,T)_{h T_i} \in X_{h T_i}$ for any $T$-tile-translate $h T_i$.
  \item $M(y,T) \cdot e^T = y \cdot e^T$.
  \item $M(y \cdot e^T,T) = M(y,T)$.
  \end{enumerate}
\end{claim}
The first property means that the projection of $M(y,T)$ on a
tile-translate $h T_i$ coincides with the projection to $h T_i$ of
some $x \in X$.  The second property means that $M$ leaves $y$
unchanged on the exterior. The third means that $M$ depends only on
the values of $y$ on the exterior.
\begin{proof}
  Let $(y,T)$ satisfy the claim hypothesis. The first property
  follows immediately from~\eqref{eq:M-def-3} and the first
  property of local completion maps. The second property is a
  restatement of~\eqref{eq:M-def-1}.

  To see the third property, consider two cases: that $g$ is in the
  $r$-exterior of $T$, and that it is in the interior. In the
  first case, 
  \begin{align*}
    [M(y \cdot e^T,T)](g) = [y \cdot e^T](g) = y(g) = [M(y,T)](g),
  \end{align*}
  where the first equality follows from~\eqref{eq:M-def-1}
  (since $g$ is in the exterior), the second from the definition of $y
  \cdot e^T$ and the third again from~\eqref{eq:M-def-1}.
  
  In the second case, $g$ is in the $r$-interior of $T$, and
  therefore an element of a tile-translate $h T_i$. Then by~\eqref{eq:M-def-3}
  \begin{align*}
    [M(y,T)](g) = [m_i(y_{h T_i},e^T_{h T_i})](g)
  \end{align*}
  where $m_i$ is the local $(X,T_i)$-completion map used to construct
  $M$. By the third property of local completion maps this can be
  written as
  \begin{align*}
    [M(y,T)](g)   &= [m_i(y_{h T_i} \cdot e^T_{h T_i},e^T_{h T_i})](g)\\
    &= [m_i((y \cdot e^T)_{h T_i},e^T_{h T_i})](g)\\
    &= [M(y \cdot e^T,T)](g).
  \end{align*}  
\end{proof}

Our approximating shift is simply going to be $Y = M(X \times Q)$, the
image of $X \times Q$ under $M$, for an appropriate choice of $Q$.

To control the entropy of $Y = M(X \times Q)$, we first prove the
following claim. Define the map $\delta_*$ by
\begin{align*}
  \delta_* \colon A^G \times \cT^G &\longrightarrow A^G
  \times \cT^G\\
  (y,T)&\longmapsto (\delta(y,T),T).
\end{align*}
That is, $\delta_*$ performs the same operation as $\delta$, but also
returns $T$, in a new, second coordinate. The next claim shows that
first deleting and then completing is the same as just completing.
\begin{claim}
  \label{clm:M-delta}
  For all $(x,T) \in X \times Q$ it holds that $M(x,T) = [M
  \circ \delta_*](x,T)$. It follows that $M(X \times Q)
  = [M \circ \delta_*](X \times Q)$.
\end{claim}
\begin{proof}
  Fix $(x,T) \in X \times Q$, and recall that $\delta(x,T) = x
  \cdot e^T$. It then follows from the third part of
  Claim~\ref{clm:M-props} that $M(x,T) = M(x \cdot
  e^T, T) = M(\delta(x,T),T)$, and so
  \begin{align*}
    M(x,T) = [M \circ \delta_*](x,T).
  \end{align*}
\end{proof}
  Since
\begin{align*}
  \delta_*(X \times Q) \subseteq \delta(X \times Q) \times Q,
\end{align*}
it follows that
\begin{align}
  \label{eq:delta-ent}
  h(\delta_*(X \times Q)) \leq h(\delta(X \times Q))  + h(Q).
\end{align}
We can therefore now show that the entropy of $Y$ is small.
\begin{claim}
  \label{clm:entropy-decomposition}
  $h(Y) \leq h(\delta(X \times Q))+h(Q)$.
\end{claim}
\begin{proof}
  \begin{align*}
    h(Y) &=h(M(X \times Q))\\
    &= h([M \circ \delta_*](X \times Q))\\
    &\leq h(\delta_*(X \times Q))\\
    &\leq h(\delta(X \times Q))+h(Q),
  \end{align*}
  where the second equality is Claim~\ref{clm:M-delta}, the next
  inequality is a consequence of the fact that factors decrease
  entropy, and the last is~\eqref{eq:delta-ent}.
\end{proof}

We next show that $M(X \times Q)$ is a good approximation of
$X$. Recall that $r \in \N$ is a parameter in the construction of $M$.
\begin{proposition}
  \label{prop:Y-good-approx}
  Let $Q$ be any quasi-tiling shift, and let $Y = M(X \times Q)$. Then
  $Y_{B_r}=X_{B_r}$.
\end{proposition}
\begin{proof}
  We will show containment in both directions. First, let $x \in
  X$. There exists a $T \in Q$ such that $\partial_r h T(h)$
  contains $B_r$, by the shift-invariance of $Q$. Hence $B_r$ is
  contained in the $r$-exterior of $T$. By the definition of
  $M$, $M(x,T)_{B_r} = x_{B_r}$, and so there exists a $y =
  M(x,T) \in Y$ such that $y_{B_r}=x_{B_r}$. Hence $X_{B_r}
  \subseteq Y_{B_r}$.

  Now, let $y =M(x,T) \in Y$. $B_r$ intersects the $r$-interior of
  at most one $T$-tile-translate, by the definition of the
  $r$-interior. If $B_r$ is contained in the $r$-exterior of $T$
  then $y$ agrees with $x$ on $B_r$ by the second part of
  Claim~\ref{clm:M-props}, and then $x_{B_r} = y_{B_r}$. Otherwise,
  $B_r$ intersects the interior of some tile-translate $h T_i$. In this
  case it must be contained in $h T_i$, and so $y_{B_r}$ is the
  projection to $B_r$ of $M(x,T)_{h T_i}$. But the latter is in
  $X_{h T_i}$ by the first part of Claim~\ref{clm:M-props}, and so
  $y_{B_r} \in X_{B_r}$. Hence $Y_{B_r} \subseteq X_{B_r}$.
\end{proof}

We are now ready to prove Proposition~\ref{prop:low-entropy-approx},
the main result of this section.
\begin{proof}[Proof of Proposition~\ref{prop:low-entropy-approx}]
  Let $X \in \shifts^\cF$, $\eps > 0$ and $r \in \N$. Define the deletion
  map $\delta$ and the global completion map $M$ as above, using the
  parameter $r$. Let $\cT$ be a tile set, and let $Q$ be a strongly
  irreducible $\cT$-quasi-tiling shift with entropy at most $\eps/2$,
  and $r$-exterior density at most $\eps'$, where
  \begin{align*}
    3\eps'\log\frac{|A|}{\eps'} \leq \eps/2.
  \end{align*}
  The existence of such a $Q$ is guaranteed by
  Theorem~\ref{cor:quasi-tilings}. Let $Y = M(X \times
  Q)$. By Claim~\ref{clm:entropy-decomposition} we have that
  \begin{align*}
    h(M(X \times Q)) &\leq h(\delta(X \times Q))+h(Q)    \\
    &\leq h(\delta(X \times Q))+\eps/2.
  \end{align*}
  By Proposition~\ref{prop:delta-low-ent} the first addend is at most
  $3\eps'\log\frac{|A|}{\eps'} \leq \eps/2$, and we have shown that
  $h(Y) \leq \eps$. Furthermore, by
  Proposition~\ref{prop:Y-good-approx}, $Y_{B_r}=X_{B_r}$.

  Since $X \in \cF$ and $Q$ is strongly irreducible, $X \times Q \in
  \cF$. By Claim~\ref{clm:m-cont} $M$ is continuous and
  shift-equivariant. Hence $Y = M(X \times Q) \subset A^G$ is also in
  $\shifts^\cF$. It follows that for every $\eps,r$ there exists a $Y
  \in \shifts^\cF$ with $h(Y) <\eps$ and $Y_{B_r}=X_{B_r}$, and so the
  set of shifts in $\cF$ with entropy in $[0, \eps)$ is dense in
  $\shifts^\cF$.
\end{proof}

\section{Fixed entropy approximation}
\label{sec:fixed-ent}
In this section we prove the following proposition. We then deduce
from it the main result of this paper, Theorem~\ref{thm:main}.
\begin{proposition}
  \label{prop:main-technical}
  Let $\cF$ be a strongly-stable class of shifts, let $c \geq 0$ and
  let $\eps > 0$. The set of shifts in $\cF$ with entropy in $[c,
  c+\eps)$ is dense in $\shifts^\cF_c$, the set of shifts in $\cF$
  with entropy at least $c$.
\end{proposition}

\subsection{Overview}

Our strategy is the following. Given a shift $X$, we construct for
every $r \in \N$ and $\eps > 0$ a sequence of shifts
$X^0,X^1,\ldots,X^\ell$ with the following properties:
\begin{enumerate}
\item $X^j \in \shifts^\cF$.
\item $X^j_{B_r}=X_{B_r}$; these are good approximations of $X$.
\item $h(X^0) \leq \eps$. 
\item $h(X^j) - h(X^{j-1}) \leq \eps$.
\item $X \subseteq X^\ell$ and so $h(X^\ell) \geq h(X)$.
\end{enumerate}
It follows that for all $0 \leq c \leq h(X)$ and some $j$, $h(X^j) \in
[c,c+\eps)$. Since this $X^j$ is a good approximation of $X$, it
follows that $X$ can be arbitrarily well approximated by shifts with
entropy in $[c,c+\eps)$, which implies
Proposition~\ref{prop:main-technical}.

To construct $X^0$ we simply apply
Proposition~\ref{prop:low-entropy-approx}, using an appropriate
quasi-tiling shift $Q$ with low entropy. To construct the rest of
these shifts, we first define a quasi-tiling shift $Q_p$, related to
the shift $Q$, and which also has low entropy. We then construct a
sequence of shifts $X_Q^1,X_Q^2,\ldots,X_Q^\ell$ such that each
$X_Q^{j+1}$ is a factor of $X_Q^j \times Q_p$. It follows that
$h(X^{j+1}_Q) \leq h(X^j_Q)+h(Q_p)$. The shifts $X^j$ are each a factor of
$X_Q^j$, and we show that $h(X_Q^j)-h(Q) \leq h(X^j) \leq h(X_Q^j)$. It thus
follows that $h(X^j) - h(X^{j-1}) \leq h(Q)+h(Q_p)$.


\subsection{$Q$ and $Q_p$}
Given a $\cT$-quasi-tiling shift $Q$, let $Q_p$ be the $\cT$-quasi-tiling
shift which includes the quasi-tilings in $Q$, with some (or no)
tile-translates removed:
\begin{align*}
  Q_p = \{T'\,:\,\exists T \in Q \mbox{ s.t. } \forall g\in
  G,T'(g)=T(g) \mbox{ or } T'(g)=\emptyset\}.
\end{align*}
Alternatively, $Q_p$ is the set of quasi-tilings bounded from above
(according to the inclusion relation) by some quasi-tiling in $Q$.
The next proposition implies that the entropy of $Q_p$ can be
controlled by an appropriate choice of $Q$.
\begin{proposition}
  \label{prop:Q-p-low-ent}
  Let $\cT=(T_1,\ldots,T_n)$ be an $(r,\eps)$-invariant tile set, and
  let $Q$ be a $\cT$-quasi-tiling. Then $h(Q_p) \leq
  h(Q)+\frac{2}{r}\log(3r)$.
\end{proposition}
\begin{proof}
  To prove this proposition, we construct $Q_p$ somewhat
  differently. Given $Q$, let $Z \subseteq \{0,1\}^G$ be the shift
  given by
  \begin{align*}
    Z = \{z \in \{0,1\}^G\,:\,\exists T \in Q\mbox{ s.t. } z(g)=1
    \mbox{ implies } T(g)\neq\emptyset\}
  \end{align*}
  and define
  \begin{align*}
    \iota \colon \cT^G \times \{0,1\}^G \to \cT^G
  \end{align*}
  by
  \begin{align*}
    [\iota(T,z)](g) =
    \begin{cases}
      T(g)&\mbox{if }z(g)=1\\
      \emptyset&\mbox{otherwise}.
    \end{cases}
  \end{align*}
  It is straightforward to verify that $Q_p = \iota(Q \times X)$. It
  follows that
  \begin{align}
    \label{eq:Q_p-Q}
    h(Q_p) \leq h(Q) + h(X).
  \end{align}
  
  Since $\cT$ is $(r,\eps)$-invariant, each of the tiles in $\cT$ is
  of size at least $r$. Hence, for any sufficiently invariant $F
  \subset G$, and for any $T \in Q$, the support of any $T_F$ is of
  size at most $2/r|F|$. It follows that the same applies to $Z$. by
  Lemma~\ref{lem:combo-low-ent}, $h(X) \leq
  \frac{2}{r}\log(3r)$. Hence, by~\eqref{eq:Q_p-Q}, $h(Q_p) \leq
  h(Q)+\frac{2}{r}\log(3r)$.
\end{proof}

The next claim follows immediately from the definitions of $Q_p$ and
strong irreducibility.
\begin{proposition}
  \label{prop:Q_p-strongly-mixing}
  If $Q$ is strongly irreducible then so is $Q_p$.
\end{proposition}

\subsection{Enumerating the completion maps}
As in the previous section, fix $r \in \N$. For each tile $T_i$ there
only exist a finite number of local $(X,T_i)$-completion maps (with
parameter $r$), since they map a finite set to a finite set. Hence
there only exists a finite number $\ell$ of global completion maps,
since each corresponds to a choice of local completions maps
$m_1,\ldots,m_n$. Enumerate them by $M^1,M^2,\ldots,M^\ell$.

For each global completion map $M^j \colon A^G \times \cT^G
\to A^G$ we define the corresponding map
\begin{align*}
  N^j \colon A^G \times \cT^G \times \cT^G &\longrightarrow
  A^G
\end{align*}
which, given $(y,T,S)$, applies $M^j$ to the configuration on a
tile-translate $h T_i$ only if $T(h) = S(h)$, and leaves the
configuration unchanged elsewhere. Formally, for each tile-translate
$h T_i$ that is both in $T$ and in $S$ we set
\begin{align*}
  N^j(y,T,S)_{h T_i} = M^j(y,T)_{h T_i},
\end{align*}
and for $g$ outside such tile-translates we set 
\begin{align*}
  [N^j(y,T,S)](g) = y(g),
\end{align*}
$N^j$ applies the same local completion maps that $M^j$ does, but it
only does so for tile-translates on which $T$ and $S$ agree. And
it leaves the configuration elsewhere unchanged.  It is immediate from
this definition that $N^j$ is continuous and shift-equivariant. It also
follows that $N^j(y,T,S) \cdot e_T = y \cdot e_T$, in
analogy to a property of $M^j$.

\subsection{The shifts $X^j$}
For $(x,T) \in X \times Q$, let
\begin{align*}
  x^0 = M^1(x,T).
\end{align*}
Given a $S^1 \in Q_p$, let
\begin{align*}
  x^1 = N^1(x^0,T,S^1).
\end{align*}
Likewise, given an additional $S^2 \in Q_p$, let
\begin{align*}
  x^2 = N^2(x^1,T,S^2).
\end{align*}
repeating the same logic, given $x^{j-1}$, and $S^j$, $1 \leq j
\leq \ell$, let
\begin{align*}
  x^j = N^j(x^{j-1},T,S^j).
\end{align*}
The next claim follows immediately from this definition, and from the
fact that $N^j(\cdot,T,\cdot)$ leaves the configuration unchanged
on the $r$-exterior of $T$.
\begin{claim}
  \label{clm:x-j}
  For all $0 \leq j \leq \ell$ it holds that $x^j \cdot e_T = x
  \cdot e_T$.
\end{claim}

Denote by $X^j$ the set of all $x^j$ that can be thus constructed, by
applying the preceding maps to some $x \in X$, $T \in Q$, and
$S^1,\ldots,S^j \in Q_p$.
\begin{claim}
  $X^j$ is a shift. Furthermore, if $X$ is in some strongly-stable
  class $\cF$ and $Q$ is strongly irreducible, then $X^j \in \cF$.
\end{claim}
\begin{proof} 
  Define
  \begin{align*}
    M_*^j \colon A^G \times \cT^G 
    &\longrightarrow
    A^G \times \cT^G\\
    (y,T) &\longmapsto (M^j(y,T),T).
  \end{align*}
  and likewise
  \begin{align*}
    N_*^j \colon A^G \times \cT^G \times \cT^G
    &\longrightarrow
    A^G \times \cT^G\\
    (y,T,S) &\longmapsto (N^j(y,T,S),T).
  \end{align*}
  Let
  \begin{align*}
    X^1_Q = M^1_*(X \times Q) \subset A^G \times Q,
  \end{align*}
  and for $1 < j \leq \ell$ let
  \begin{align*}
    X^j_Q = N^j_*(X^{j-1}_Q \times Q_p).
  \end{align*}
  Then $X^j = N^j(X^{j-1}_Q \times Q_p)$; equivalently, it is the
  projection to the first coordinate of $X^j_Q$, and is therefore a
  shift. Since $Q$ is strongly irreducible then so is $Q_p$, by
  Proposition~\ref{prop:Q_p-strongly-mixing}. Finally, since $X^j$ is
  constructed by taking a series of factors and products involving
  $X$, $Q$ and $Q_p$, it follows from the definition of
  strongly-stable classes that $X^j$ is also in $\cF$.
\end{proof}

\begin{proposition}
  \label{prop:X-j}
  For $0 < j \leq \ell$
  \begin{enumerate}
  \item $h(X^j) \leq h(X^{j-1})+H(Q)+H(Q_p)$.
  \item $X_{B_r} = X^j_{B_r}$.
  \end{enumerate}
\end{proposition}
\begin{proof}
  \begin{enumerate}
  \item Recall that $X^j_Q = N^j(X^{j-1}_Q \times Q_p)$. Recall also
    that the projections of $X^{j-1}_Q$ to the first and second
    coordinates are $X^{j-1}$ and $Q$, respectively.  Hence
    \begin{align*}
      h(X^j) &\leq h(X^{j-1}_Q) + h(Q_p)\\
      &\leq h(X^{j-1})+h(Q)+h(Q_p).
    \end{align*}
  \item $X^0_{B_r} = X_{B_r}$ by
    Proposition~\ref{prop:Y-good-approx}. Fix $(x,T) \in X \times
    Q$, and let $x^0=M^1(x,T)$, $x^1=N^1(x^0,T,S^1)$,
    etc. Recall that $x^j \cdot e_T = x \cdot e_T$
    (Claim~\ref{clm:x-j}); the maps $N^j$ do not alter the
    configuration on the $r$-exterior of $T$. Now, inside each
    tile-translate applying $N^j$ either does nothing, or else is the same
    as applying $M^j$. Hence the interiors of the tile-translates are also
    compatible with $X$, and the same argument of
    Proposition~\ref{prop:Y-good-approx} applies in this case too.
  \end{enumerate}
\end{proof}

We would next like to show the following proposition.
\begin{proposition}
  \label{prop:X-ell}
  $X \subseteq X^\ell$.
\end{proposition}
\begin{proof}
  To prove this proposition, we construct for each $x \in X$ a
  sequence of quasi-tilings $\{S^j\}_{j=1}^\ell$ such that the
  associated $x^\ell$ is equal to $x$.

  To this end, fix some $x \in X$ and $T \in Q$.  Let $h T_i$ be a
  $T$-tile-translate. Then by Claim~\ref{claim:complete-x} there exists
  a local $(X,h T_i)$-completion map $m_i$ that is $(x,h
  T_i)$-compatible; that is, it completes to $x_{h T_i}$ any
  configuration on $h T_i$ that is compatible with $x$ on the exterior
  of $h T_i$.

  Now, there will an $N^j$ that will use $m_i$ to complete translates of
  $T_i$. Hence for such an $N^j$ it will hold that
  \begin{align}
    \label{eq:compatible}
    M^j(y,T,T')_{h T_i} = x_{h T_i}
  \end{align}
  for all $T'$ such that $T'(h)=T(h)=i$ and for all $y \in
  A^G$ such that $(y \cdot e_{T}) = (x \cdot e_{T})$. In this
  case we say that $N^j$ is {\em $(x,h T_i)$-compatible}.

  Let $C^j_{x,T} \subset G$ be the set of $T$-tile-translate
  locations $h$ for which $N^j$ is $(x,h T(h))$-compatible:
  \begin{align*}
    C^j_{x,T} = \left\{h \in G\,:\, N^j\mbox{ is
        $(x,h T(h))$-compatible}\right\}.
  \end{align*}
  Then for each $T$-tile-translate $h T_i$ there is some $1 \leq j \leq \ell$
  such that $h \in C^j_{x,T}$. We furthermore choose the sets
  $C^j_{x,T}$ so that each such $h$ appears in exactly one set.

  Let $S^j$ be given by
  \begin{align*}
    S^j(h) =
    \begin{cases}
      T(h)&\mbox{ if }h \in C^j_{x,T}\\
      0&\mbox{otherwise}
    \end{cases}.
  \end{align*}
  Clearly $S^j \in Q_p$. It follows immediately from the definition
  of $C^j_{x,T}$ that if $h \in C^j_{x,T}$ then
  \begin{align*}
    N^j(y,T,S^j)_{h T(h)} = x_{h T(h)}.
  \end{align*}
  for all $y \in A^G$ such that $y \cdot e_{T} = x \cdot e_{T}$.

  Following our construction of the shifts $X^j$, let
  \begin{align*}
    x^0 &= M^1(x,T)\\
    x^1 &= N^1(x^0,T,S^1)\\
    x^2 &= N^2(x^1,T,S^2)\\
    \vdots\\
    x^\ell &= N^\ell(x^{\ell-1},T,S^\ell).
  \end{align*}
  By the definition of the configurations $S^j$, for each
  $T$-tile-translate $h T_i$ there will be a $j$ such that $S^j(h)
  = T(h)$. Hence for that $j$ it will holds that
  \begin{align*}
    x^j_{h T_i} = N^j(x^{j-1},T,S^j)_{h T_i} = x_{h T_i}.
  \end{align*}
  For all $j'>j$ the configuration on this tile will remain unchanged,
  and so $x^\ell_{h T_i}=x_{h T_i}$. Since this holds for all
  tile-translates, and since $x^\ell$ agrees with $x$ on the $r$-exterior
  of $T$ (Claim~\ref{clm:x-j}), it follows that $x^\ell = x$. Hence
  $x \in X^\ell$.
\end{proof}

We are now ready to prove Proposition~\ref{prop:main-technical}.
\begin{proof}[Proof of Proposition~\ref{prop:main-technical}]
  Let $X \in \shifts^\cF$, $\eps > 0$ and $r \in \N$. Define maps
  $M^j$ as above, using the parameter $r$.

  Let $\cT$ be a tile set, and let $Q$ be a strongly irreducible
  $\cT$-quasi-tiling shift such that $h(Q) + h(Q_p) < \eps$ and such
  that if $X^0=M^1(X \times Q)$ then $h(X^0) < \eps$. The existence of
  such a quasi-tiling shift is guaranteed by
  Theorem~\ref{cor:quasi-tilings} and
  Propositions~\ref{prop:low-entropy-approx}
  and~\ref{prop:Q-p-low-ent}, for $r$ large enough.

  Let $X^1,X^2,\ldots,X^\ell$ be the shifts defined above. They are in
  $\shifts^\cF$, since they are constructed from $X$ and $Q$ by taking
  products and factors only.

  By Proposition~\ref{prop:X-ell} $X \subseteq X^\ell$, and so
  $h(X^\ell) \geq h(X)$. On the other hand, by the first part of
  Proposition~\ref{prop:X-j}, $h(X^j) - h(X^{j-1}) \leq \eps$. It
  follows that $h(X^j) \in [c,c+\eps)$ for some $1 \leq j \leq
  \ell$. Finally, by the second part of Proposition~\ref{prop:X-j},
  $X^j_{B_r} = X_{B_r}$.
\end{proof}

The proof of our main theorem is now straightforward.
\begin{proof}[Proof of Theorem~\ref{thm:main}]
  We first prove that for all $\epsilon$ and all $c > 0$ the set of shifts
  with entropy in the range $[c,c+\epsilon)$ is dense in the set of
  shifts with entropy greater than or equal to $c$.

  Let $X$ be a shift with entropy $h(X)\ge c$ and fix $\eps$.  Let
  $Y_n$ be a subshift that agrees with $X$ on $B_r$ and has entropy in
  $[c,c+\eps)$; the existence of these subshifts is guaranteed by
  Proposition~\ref{prop:main-technical}. Since the sets $B_r$ exhaust
  $G$, it follows that $\lim_n Y_n = X$.  Thus the set of shifts with
  entropy in $[c,c+\eps)$ is dense in $\shifts_{\geq c}^\cF$.

  By Proposition~\ref{prop:semi-cont-ent} the entropy function is
  upper semi-continuous. Hence the set of shifts with entropy in the
  range $[0,c+\eps)$ is open, for all $c+\eps>0$. Thus it follows that
  the set of shifts with entropy in the range $[c,c+\eps)$ is open in
  $\shifts_{\geq c}^\cF$.

  Since $\shifts_c^\cF$ is the intersection of the sets of shifts with
  entropy in the range $[c,c+1/n)$ (for $n$ in $\N$), and since, by
  the above, each of these is open and dense in $\shifts_{\geq c}^\cF$,
  then by the Baire Category Theorem $\shifts_c^\cF$ is comeagre in
  $\shifts_{\geq c}^\cF$.
\end{proof}

\section{The topology of $\shifts_{\geq c}$}
\label{sec:topology}

In this section we prove Theorem~\ref{thm:isolated}.
\isolated*

We first note the following fact.
\begin{proposition}
  If $X$ is a strongly irreducible shift of finite type with $h(X)=c$
  then $X$ is an isolated point in $\shifts_{\geq c}$.
\end{proposition}
\begin{proof}
  Let $\lim_n X^n =X$ with $X^n\neq X$ for all $n$. Since $X$ is of
  finite type, all but finitely many of the $X^n$ have to be proper
  subshifts of $X$. Since $X$ is strongly irreducible, each of its
  proper subshifts have entropy strictly less than
  $c$~\cite{ceccherini2012myhill}*{Proposition 4.2}. Hence all but
  finitely many of the $X^n$ are outside $\shifts_{\geq c}$, and $X$ is an
  isolated point in $\shifts_{\geq c}$.
\end{proof}

In light of this proposition, we prove Theorem~\ref{thm:isolated} by
finding a countable family of strongly irreducible shifts of finite
type, whose entropies form a dense set in $[0,\log|A|]$. To this end,
we employ a strategy similar to the one used in
Section~\ref{sec:fixed-ent}: for every $\eps > 0$ we construct a
sequence of shift $X^{1},X^{2},\ldots,X^{\ell}$, where
\begin{enumerate}
\item Each $X^{j}$ is a strongly irreducible shift of finite type.
\item $X^{j-1} \subset X^{j}$.
\item $h(X^{0}) = 0$.
\item $h(X^{j}) - h(X^{j-1}) \leq \delta(\eps)$, where $\lim_{\eps \to 0}\delta(\eps)=0$.
\item $h(X^{\ell}) = \log|A|$.
\end{enumerate}
It follows that the set of entropies $\{h(X^{j})\}$ is dense in $[0,\log|A|]$.

Our proof proceeds as follows: for every $\eps$ we choose a
sufficiently good tile set $\cR$. The shift $X^{j}$ will be the shift
which is supported on at most an $j / |R_1|$ proportion of any
$\cR$-tile-translate, where $R_1$ is the largest tile in $\cR$.

This is clearly a strongly irreducible shift of finite type. It is
immediate that $h(X^{|R_1|}) = \log|A|$, and we
show that $h(X^{j})$ is low when $j$ is low.

To show that $h(X^{j}) - h(X^{j-1}) \leq \eps$ for large $j$ we again
use a strategy similar to that of Section~\ref{sec:fixed-ent}: We show
that $X^{j}$ is a subshift of a factor of $X^{j-1} \times Y$,
where $Y$ is some shift of entropy at most $\delta(\eps)$. This implies
that $h(X^{j}) \leq h(X^{j-1} \times Q) = h(X^{j-1})+\delta(\eps)$.

\subsection{Compatible quasi-tilings}
Before defining the shifts $X^j$, we take a short intermission to
define a technical tool which will be useful to that end, and state a
simple claim regarding it.
\begin{definition}
  \label{def:compatible}
  Let $\cT = (T_1,\ldots,T_n)$ and $\cR = (R_1,\ldots,R_m)$ be tile sets,
  and fix $\eps > 0$. A $\cT$-quasi-tiling $T$ is said to be
  $(\cR,\eps)$-compatible if, for any tile-translate $K = h R_i$ it holds that
  $e(T,K) \leq \eps|K|$.
\end{definition}

Given tile sets $\cT$ and $\cR$, and given an $\eps>0$, the set of
$(\cR,\eps)$-compatible $\cT$-quasi-tiling is a shift. In fact, it is
a shift of finite type. The next claim shows that it is non-empty, if
$\cR$ is sufficiently good. It follows immediately from the definitions.
\begin{claim}
  \label{clm:compatible}
  Let $\cT$ be an $\half \eps$-good tile set, and let $\cR$ be a tile
  set with $\rho_1R_i$ sufficiently small.  Then there exists a
  $\cT$-quasi-tiling that is $(\cR,\eps)$-compatible.
\end{claim}

\subsection{Constructing the shifts $X^{j}$}
Fix $\eps > 0$. Let $\cT$ be an $\half \eps$-good tile set, and let
$\cR$ be an $\eps$-good tile set such that $\rho_1R_i$ is small enough
so that, by Claim~\ref{clm:compatible}, there exist
$(R,\eps)$-compatible $\cT$-quasi-tilings. Furthermore, let each tile
in $\cR$ be of size at least $1/\eps$, and be
$(r(\cT),\eps)$-invariant. Let $Q$ be the $\cT$-quasi-tiling shift of
tilings with these properties.

Let $R_1$ be a largest tile in $\cR$. For every $0 \leq j \leq |R_1|$
let $X^j \subseteq A^G$ be the shift of all $x \in A^G$ whose
projections to any $\cR$-tile-translate $h R_i$, $x_{h R_i}$, have
support of size at most $j |R_i|/|R_1|$.

We next show that $X^j$ has low entropy for low values of $j$.
\begin{claim}
  \label{clm:X-j-low-ent}
  $h(X^j) \leq \frac{j}{|R_1|}\log\frac{3|A||R_1|}{j}+\eps\log|A|$.
\end{claim}
\begin{proof}
  Denote $\alpha = j/|R_1|$, so that the support of $X^j_{R_i}$ is of
  size at most $\alpha|R_i|$. It follows from
  Lemma~\ref{lem:combo-finite-set} that $X^j_{R_i}$ is of size at most
  $(3|A|/\alpha)^{\alpha |R_i|}$. We can now apply
  Proposition~\ref{prop:tiling-entropy-bound} to $X^j$ and $\cR$, with
  $p = (3|A|/\alpha)^\alpha$. This yields
  \begin{align*}
    h(X^j) \leq \alpha\log\frac{3|A|}{\alpha}+\log(|A|)\eps.
  \end{align*}
\end{proof}

\subsection{Constructing the shift $Y$}
Fix $\eps > 0$.  Recall that $Q$ is the $\cT$-quasi-shift of
$(R,\eps)$-compatible $\cT$-quasi-tilings.

Let $Y^0 \subset Q \times A^G$ be a shift defined as follows: if
$(T,y) \in Y$ then for every tile-translate $h T_i$ in $T$ it holds
that $y(g) \neq 0$ for at most one $g \in h T_i$. For $g$ outside the
tile-translates $y(g)$ can take any value.

Let $Y \subset A^G$ be the projection of $Y^0$ on its second
coordinate. We can bound the size of the support of $Y_F$ by the
number of tile-translates fully contained in a finite set $F$, plus
the size of $F$'s boundary $\partial_{r(\cR)}F$. Since each tile is of
size at least $1/\eps$, the total number of tiles is at most
$\eps|F|$. If we choose $F$ so that $\rho_{r(\cR)}F$ is small enough,
then the size of $\partial_{r(\cR)}F$ will be at most $\eps|F|$. It
follows that the support of $Y_F$ is of size at most
$2\eps|F|$. Hence, by Lemma~\ref{lem:combo-low-ent} we have that
\begin{align}
  \label{eq:Y-ent}
  h(Y) \leq 2\eps\log\frac{|A|}{\eps}.
\end{align}

\subsection{Realizing entropies}
Let $\varphi \colon A^G \times A^G \to A^G$ be given by
\begin{align*}
  [\varphi(x,y)](g) =
  \begin{cases}
    x(g)& \mbox{if }x(g) \neq 0\\
    y(g)& \mbox{otherwise}.
  \end{cases}
\end{align*}

\begin{claim}
  If $j / |R_1| > 3\eps$ then $X^{j}$ is a subshift of
  $\varphi(X^{j-1} \times Y)$.
\end{claim}
\begin{proof}
  Let $x^j \in X^{j}$. Choose a $T \in Q$. Let $x^{j-1} \in A^G$
  vanish outside the $T$ tile-translates. Inside each $T$-tile-translate
  $h T_i$ let $x^j$ coincide with $x^{j-1}$, except for at a single
  point $g \in h T_i$ in which $x^j(g) \neq 0$ (assuming one exists),
  where we set $x^{j-1}(g)=0$. Let $L$ be the set of these locations $g$.

  Let $y \in A^G$ be given by $y(g)=x^j(g)$ outside the
  $T$-tile-translates, and let $y(g)=0$ inside the tile-translates, except
  for $g \in L$, where we set $y(g) = x^j(g)$. Hence $y \in Y$, since
  $L$ intersects each $T$-tile-translate in at most one element. It is
  immediate that $\varphi(x^{j-1},y)=x^j$. The claim will thus be
  proved if we show that $x^{j-1}\in X^{j-1}$.

  To see that $x^{j-1}\in X^{j-1}$, we fix a an $\cR$-tile-translate
  $h R_i$, and consider two cases.  First, if the support of $x^j_{h
    R_i}$ is of size less than $j|R_i|/|R_1|$ (i.e., the maximum it
  can be in $X^j$), then $x^j_{h R_i} \in X^{j-1}_{h R_i}$, and hence
  $x^{j-1}_{h R_i} \in X^{j-1}_{h R_i}$, since its support is at most
  that of $x^j_{h R_i}$.

  Second, consider the case that the support of of $x^j_{h R_i}$ is of
  size $j|R_i|/|R_1|$, and so it is larger than $3\eps|R_i|$. Each
  $\cR$-tile-translate $h R_i$ is well covered by $T$, since $T$ is
  $(\cR,\eps)$-compatible; in particular, $e(T, h R_i) \leq \eps
  |R_i|$. Since $\rho_{r(\cT)}(R_i)\le \eps$, it follows that the
  union of the $T$-tile-translates that are {\em fully contained} in
  $h R_i$ is at least of size $(1-2\eps)|R_i|$. Since the support of
  $x^j_{h R_i}$ is of size $3\eps|R_i|$ it follows that at least one
  $T$-tile-translate has an element $g$ such that $x^j(g) \neq
  0$. This will be removed in $x^{j-1}$, and so the size of the
  support of $x^{j-1}_{h R_i}$ is at most $j|R_i|/|R_1| - 1 \leq
  (j-1)|R_i|/|R_1|$. Hence $x^{j-1} \in X^{j-1}$.
\end{proof}

It follows that if $j / |R_1| > 3\eps$ then
\begin{align*}
  h(X^j)-h(X^{j-1}) \leq h(Y) \leq 2\eps\log\frac{|A|}{\eps},
\end{align*}
where the second inequality is~\ref{eq:Y-ent}.

It follows from Claim~\ref{clm:X-j-low-ent} that if $j / |R_1| \leq
3\eps$ then
\begin{align*}
  h(X^j) - h(X^{j-1}) \leq h(X^j) \leq
  3\eps\log\frac{|A|}{\eps}+\log(|A|)\eps.
\end{align*}
Hence, if we set
\begin{align*}
  \delta(\eps) =3\eps\log\frac{|A|}{\eps}+\log(|A|)\eps
\end{align*}
then we have shown that for all $0 \leq j\leq \ell$
\begin{align*}
  h(X^j) - h(X^{j-1}) \leq \delta(\eps),
\end{align*}
thus proving Theorem~\ref{thm:isolated}.

\appendix
\section{Combinatorial lemmata for bounding entropy}

In this appendix we prove two combinatorial lemmata that are useful
for bounding entropy.
\begin{lemma}
  \label{lem:combo-finite-set}
  Let $X \in A^G$ be a shift, and let $0$ be a distinguished element
  of $A$. Suppose that for some finite $F \subset G$ it holds that for
  every $x \in X$ the projection $x_F$ vanishes (i.e., equals $0$) on
  all but at most $\eps|F|$ of the elements of $F$. Then the size of
  the projection of $X$ on $F$ can be bounded by
  \begin{align*}
    |X_F| \leq \left(\frac{3|A|}{\eps}\right)^{\eps|F|}.
  \end{align*}
\end{lemma}
\begin{proof}
  Without loss of generality, we may assume that $\eps|F|$ is an
  integer.  Let $f \colon F \to A$ be the projection of $x \in X$ to
  $F$. To choose $f$, we can first choose an $H \subset F$ of size
  $\eps|F|$ which contains the support of $f$.  Given such an $H$
  there are $A^{\eps|F|}$ functions $f$ supported on $H$. Since there
  are $\binom{|F|}{\eps |F|}$ different choices of $H$, it follows that
  there are at most $\binom{|F|}{\eps |F|}A^{\eps|F|}$ distinct
  functions.

  A standard bound on binomial coefficients  which follows from
  Stirling's approximation is
  \begin{align*}
    \binom{n}{k} \leq \left(\frac{n e}{k}\right)^k < \left(\frac{3n }{k}\right)^k.
  \end{align*}
  It follows that there are at most
  \begin{align*}
    \binom{|F|}{\eps |F|}A^{\eps|F|} \leq
    \left(\frac{3|F|}{\eps|F|}\right)^{\eps|F|}A^{\eps|F|} =
    \left(\frac{3|A|}{\eps}\right)^{\eps|F|}
  \end{align*}
  possible projections of $X$ on $F$.
\end{proof}

\begin{lemma}
  \label{lem:combo-low-ent}
  Let $X \in A^G$ be a shift, and let $0$ be a distinguished element
  of $A$. Suppose that for some $\delta>0$ and every finite,
  $(1,\delta)$-invariant $F \subset G$, it holds that for every $x \in
  X$ the projection $x_F$ vanishes (i.e., equals $0$) on all but at
  most $\eps|F|$ of the elements of $F$. Then $h(X) \leq \eps\log
  \frac{3|A|}{\eps}$.
\end{lemma}
\begin{proof}
  Let $F \subset G$ be finite and $(1,\delta)$-invariant. By
  Lemma~\ref{lem:combo-finite-set}
  \begin{align*}
    |X_F| \leq \left(\frac{3|A|}{\eps}\right)^{\eps|F|}.
  \end{align*}
  Since this holds for all $(1,\delta)$-invariant $F$, it follows that
  \begin{align*}
    h(X) \leq \eps\log \frac{3|A|}{\eps}.
  \end{align*}
\end{proof}

\section{Semi-continuity of the entropy function}
\label{sec:semi-cont-ent}
In this appendix we prove Proposition~\ref{prop:semi-cont-ent}.
\semicont*

Our proof uses ideas which are similar to those used by Lindenstrauss
and Weiss~\cite{lindenstrauss2000mean}*{Appendix 6}

\subsection{Bounding entropy using quasi-tilings}

Recall that given a shift $X$ and a finite subset $K\in G$, $X_K$ is
the projection of $X$ onto $K$. We first state an easy lemma about
the size of this set.
\begin{lemma}
  \label{lem:finite-prod-bounds}
  Let $X$ be a shift. If $K_1,\ldots,K_n$ are finite sets and
  $K=\union_{i\in(1,\ldots,n)} K_i$ then
  \begin{align*}
    |X_K| \leq \prod_{i=1}^{n}|X_{K_i}|
  \end{align*}
\end{lemma}
\begin{proof}
  By induction is suffices to prove the lemma when $n=2$ which we will
  do. Let $K=K_1 \cup K_2$. Then any element of $X_K$ can be projected
  to an element of $X_{K_1}$ and also into an element of $X_{K_2}$, by
  restriction.  We can thus map an element of $X_K$ to the Cartesian
  product of its mappings on $X_{K_1}$ and $X_{K_2}$ .  Since this
  mapping is injective, the claim follows.
\end{proof}
We will need two simple corollaries of this lemma before proving the
main proposition.
\begin{corollary}
  \label{cor:semi-cont1}
  Let $X \subseteq A^G$ be a shift, let $K$ be finite subset of $G$
  and $H$ a subset of $K$ with $p|K|$ elements, for some $0 \leq p
  \leq 1$. Then
  \begin{align*}
    |X_K|\leq|X_H|\cdot |A|^{(1-p)|K|}
  \end{align*}
\end{corollary}
\begin{proof}
  There are $(1-p)|K|$ elements of $K \setminus H$. Since there are
  $|A|^{(1-p)|K|}$ colorings of $K \setminus H$ in the full shift,
  there are at most that many colorings of $K \setminus H$ in
  $X_K$. The corollary then follows by applying lemma
  \ref{lem:finite-prod-bounds} to the sets $K_1 = H$ and $K_2 = K
  \setminus H$.
\end{proof}

Our second corollary is an immediate consequence
Lemma~\ref{lem:finite-prod-bounds} to sets satisfying certain
conditions; these conditions will later be satisfied by
tile-translates.
\begin{corollary}
  \label{cor:semi-cont2}
  Let $X \subseteq A^G$ be a shift and suppose that for all elements
  of some collection $K_1,K_2,\ldots,K_n$ of disjoint, finite subsets
  of $G$, and for some constant $p > 0$, we have $|X_{K_i}| \leq
  p^{|K_i|}$ for $1 \leq i \leq n$. Let $K=\union_i K_i$. Then
  \begin{align*}
    |X_K| \leq p^{|K|}.
  \end{align*}
\end{corollary}

We are now ready to prove that the entropy of a shift can be
controlled, given the size of its projections to tiles of a good tile
set. We will use this proposition in this section, as well as in
Section~\ref{sec:topology}.
\begin{proposition}
  \label{prop:tiling-entropy-bound}
  Let $X \subseteq A^G$ be a shift. Let $\cT$ be a
  $\half\eps$-good tile set. Suppose that $|X_{T_i}| \leq
  p^{|T_i|}$ for some constant $p > 0$ and every tile $T_i \in
  \cT$. Then
  \begin{align*}
    h(X) \leq \log p+\log(|A|)\eps.
  \end{align*}
\end{proposition}
\begin{proof}
  Let $F$ be any finite, $(r(\cT),\half\eps)$-invariant subset of $G$.
  It follows from Theorem~\ref{thm:quasi-tiling-shifts-intro} that
  there exists a $\cT$-quasi-tiling $T$ with the following properties:
  \begin{enumerate}
  \item Each tile-translate in $T$ is fully contained in $F$.
  \item \label{item:e-leq} $e(T,F) \leq \eps|F|$.
  \end{enumerate}
  By (\ref{item:e-leq}), if we denote by $E \subseteq F$ the set of
  elements covered the tile-translates in $T$, then $|E| \geq
  (1-\eps)|F|$. 

  Fix a tiling $T$ of $F$ with the properties described above.
  Since $X$ is shift-invariant,
  \begin{align*}
    |X_{h T(h)}| \leq p^{|h T(h)|},
  \end{align*}
  for all tile-translates $h T(h)$ in $T$. Since $E$ is the union
  of all the tile-translates in $T$, and since the sum of their
  sizes is $|E|$, by applying Corollary~\ref{cor:semi-cont2} we see
  that
  \begin{align*}
    |X_E| \leq \exp\big(\log p|E|\big) \leq \exp\big(\log
    p|F|\big).
  \end{align*}
  applying Corollary~\ref{cor:semi-cont1} to $F = E \cup (F \setminus
  E)$ yields
  \begin{align*}
    |X_F| &\leq
    \exp\big(\log p|F|\big) \cdot |A|^{\eps|F|}\\
    &=\exp\big(\log p|F|+\log(|A|)\eps|F|\big).
  \end{align*}
  Hence
  \begin{align*}
    \frac{1}{|F|}\log|X_F| \leq \log p+\log(|A|)\eps.
  \end{align*}
  Since this holds for all $F$ with a small enough boundary, it
  follows that
  \begin{align*}
    h(X) \leq \log p+\log(|A|)\eps.
  \end{align*}
\end{proof}

\subsection{Proof of semi-continuity}
To prove the semi-continuity of the entropy function it suffices to
show that for every sequence $X^1,X^2,\ldots$ of shifts with limit
$\lim_n X^n=X$ it holds that $\limsup_n h(X^n) \leq h(X)$.

We will prove Proposition~\ref{prop:semi-cont-ent} by showing that for
every $\eps >0$ there is an  $N$ large enough so that
for all $n \geq N$  it holds that
\begin{align*}
  h(X^n)  \leq h(X)+\eps+\log(|A|)\eps.
\end{align*}

Fix $\eps > 0$. Let $\cT = (T_1,\ldots,T_n)$ be a $\delta$-good tile
set, with each $T_i$ being $(1,\delta)$-invariant, for some $\delta <
\half \eps$ small enough so that
\begin{align}
  \label{eq:large-enough-tiles}
  \frac{1}{|T_i|}\log|X_{T_i}| \leq h(X)+\eps
\end{align}
for all $T_i \in \cT$.

Let $N$ be large enough so that, for all $n \geq N$,
$X^n_{T_i}=X_{T_i}$ for all $T_i \in \cT$. Hence
\begin{align*}
  \frac{1}{|T_i|}\log|X^n_{T_i}| = \frac{1}{|T_i|}\log|X_{T_i}|,
\end{align*}
and by~\eqref{eq:large-enough-tiles}
\begin{align*}
  \frac{1}{|T_i|}\log|X^n_{T_i}| \leq h(X) + \eps.
\end{align*}
Rearranging, we get that for every tile $T_i$
\begin{align*}
  |X^n_{T_i}| \leq e^{(h(X)+\eps)|T_i|},
\end{align*}
and by shift-invariance, the same holds for every tile-translate $h T_i$:
\begin{align*}
  |X^n_{h T_i}| \leq e^{(h(X)+\eps)|h T_i|},
\end{align*}

We can now apply Proposition~\ref{prop:tiling-entropy-bound} to $X^n$
and $\cT$, by setting $p = e^{h(X)+\eps}$. This yields
\begin{align*}
  h(X^n) \leq h(X)+\eps+\log(|A|)\eps,
\end{align*}
proving Proposition~\ref{prop:semi-cont-ent}.

\bibliography{entropy}
\end{document}